%% file: precorn.tex
\documentclass[12pt]{amsart}

\usepackage{amsmath}
\usepackage{amssymb}
\usepackage{amsthm}
\usepackage{graphicx}
\usepackage{color}

\newcommand{\bZ}{\mathbb{Z}}

\DeclareMathOperator{\ric}{Ric}

\DeclareMathOperator{\vol}{vol}
\DeclareMathOperator{\tr}{tr}

\DeclareMathOperator{\diver}{div}

\DeclareMathOperator{\Conf}{Conf}  
\DeclareMathOperator{\Mob}{Mob}    
\DeclareMathOperator{\TO}{O^+}     

\newcommand{\suchthat}{\mathrel{}:\mathrel{}}  

\def\sideremark#1{\ifvmode\leavevmode\fi\vadjust{\vbox to0pt{\vss
 \hbox to 0pt{\hskip\hsize\hskip1em
 \vbox{\hsize3cm\tiny\raggedright\pretolerance10000
 \noindent #1\hfill}\hss}\vbox to8pt{\vfil}\vss}}}

\newtheorem{theorem}{Theorem}[section]
\newtheorem{Theorem}{Theorem}
\newtheorem{lemma}[theorem]{Lemma}

\newtheorem{proposition}[theorem]{Proposition}

\theoremstyle{definition}
\newtheorem*{remark}{Remark}

\allowdisplaybreaks[1]

\author[J. S. Case]{Jeffrey S. Case}
\address{109 McAllister Building, Department of Mathematics, Penn State University, University Park, PA 16802, USA}
\email{jscase@psu.edu}
\author[Y.-J. Lin]{Yueh-Ju Lin}
\address{Department of Mathematics, Statistics, and Physics, Wichita State University, Wichita, KS 67260, USA}
\email{yueh-ju.lin@wichita.edu}
\author[S. E. McKeown]{Stephen E. McKeown}
\address{Department of Mathematical Sciences, FO 35, University of Texas at Dallas, 800 W. Campbell Road, Richardson, TX 75080, USA}
\email{stephen.mckeown@utdallas.edu}
\author[C. B. Ndiaye]{Cheikh Birahim Ndiaye}
\address{Department of Mathematics, Howard University, Annex 3, Graduate School of Arts and Sciences \#217, Washington, DC 20059, USA}
\email{cheikh.ndiaye@howard.edu}
\author[P. Yang]{Paul Yang}
\address{Princeton University, Department of Mathematics, Fine Hall, Washington Road, Princeton, NJ 08544-1000, USA}
\email{yang@math.princeton.edu}
\subjclass[2020]{Primary 53C21; Seconary 35J40, 53C18, 35B65}
\keywords{Bilaplacian, manifolds with corners, corner regularity, Gauss-Bonnet,
$Q$-curvature, $T$-curvature, $U$-curvature}

\title[Cherrier--Escobar on the half-ball]{A fourth-order Cherrier--Escobar problem with prescribed corner behavior on the half-ball}

\begin{document}
\begin{abstract}
	We show that the half-ball in $\mathbb{R}^4$ can be conformally changed so that the only contribution to the Gauss--Bonnet formula is a constant term at the corner. 
	This may be seen as a fourth-order Cherrier--Escobar-type problem on the half-ball.
\end{abstract}
\maketitle
\input{tex/intro}
\input{tex/set}
\input{tex/sphere}
\input{tex/cons}
\input{tex/conf}

\section*{Acknowledgements}
This work was initiated and significantly advanced during the workshop \emph{Partial differential equations and conformal geometry} held at the American Institute of Mathematics (AIM) in August 2022.
We thank AIM for providing an ideal research environment.
We are grateful for helpful conversations with Matthew Gursky,
Tzu-Mo Kuo, and Andrew Waldron.

JSC was partially supported by a grant from the Simons Foundation (Grant No.\ 524601). SEM was partially supported by a grant from the Simons Foundation (Grant No. 966614).
CBD was partially supported by a grant from the National Science Foundation (Grant
No. DMS-2000164). PY was partially supported by a grant from the National Science
Foundation (Grant No. DMS-1509505).

\bibliographystyle{alpha}
\bibliography{precorn}
\end{document}

%% file: tex/intro.tex
\section*{Introduction}

The celebrated Uniformization Theorem states that every compact Riemannian surface can be conformally changed to one of constant Gauss curvature. Since the Gauss--Bonnet theorem relates the Euler characteristic and the
total Gauss curvature, this constant is determined up to scaling by the surface, 
and one interpretation of uniformization is thus that every Riemannian surface can be conformally changed to one for which the ``geometric contribution'' to the
genus is uniform. It can similarly be shown that when $M$ is a compact surface with smooth boundary, there is a conformal change for which $K = 0$ and the mean curvature is constant; 
thus, all the curvature is ``pushed to the boundary'' (and constant). This can be viewed
as a generalization of the Riemann Mapping Theorem of complex analysis.

In four dimensions, the Gauss--Bonnet formula on a closed manifold can be written in terms of the fourth-order $Q$-curvature: 
\[4\pi^2\chi(X^4) = \int_X\left(\frac{1}{8}|W|^2 + \frac{1}{2}Q_g\right)dV_g,\]
where $W$ is the Weyl curvature tensor and
\begin{equation*}
	6Q_g := -\Delta_gR_g + R_g^2 - 3\lvert\ric_g\rvert^2.
\end{equation*}
Although the order of $Q$ is higher than that of the Pfaffian, it has the distinct advantage that, like the Gauss curvature on surfaces, it transforms linearly under conformal change: if $\tilde{g} = e^{2\omega}g$, then
\[\widetilde{Q} = e^{-4\omega}(Q + P_4\omega),\]
where $P_4$ is the \emph{Paneitz operator}
\begin{equation*}
 P_4u := \Delta_g^2u + \delta_g\left(2\ric - \frac{2}{3}R_g\right)du.
\end{equation*}
Here $\ric$ is the Ricci tensor viewed as an endomorphism and $\delta_g$ is the divergence operator. The Paneitz operator is itself a conformal invariant:
\begin{equation*}
 \widetilde{P}_4 = e^{-4\omega}P_4 .
\end{equation*}
In light of these facts, a natural question is the $Q$-Yamabe problem: given a compact four-manifold $(X,g)$, does there exist $\omega$ so that $e^{2\omega}g$ has constant $Q$-curvature?
The answer is generically yes~\cite{cy95,dm08,lll12}. 
As with the Uniformization Theorem,
this has the interpretation that topological content may be distributed equally around the manifold, at least as much as possible---the Weyl curvature is a pointwise conformal invariant, so $|W|^2$ cannot generally
be made constant.

Chang and Qing~\cite{cq97i} introduced a boundary version of the Gauss--Bonnet formula for which the boundary term has good conformal invariance properties.
Specifically, on the boundary $M = \partial X$
of a four-manifold $X$ with boundary,
they defined the pointwise conformally invariant quantity
\begin{equation*}
	\mathcal{L} := \mathring{L}^{\mu\nu}R_{\mu\nu}^g -
	2\mathring{L}^{\mu\nu}R_{\mu\nu}^h + \frac{2}{3}H|\mathring{L}|_h^2 -
	\tr \mathring{L}^3
\end{equation*}
and an additional extrinsic curvature quantity
\begin{align*}
	T & := -\frac{1}{12}\mu(R_g) - \mathring{L}^{\mu\nu}R^g_{\mu\nu}
	+ \mathring{L}^{\mu\nu}R_{\mu\nu}^h - \frac{1}{2}H|\mathring{L}|_h^2
	+ \frac{2}{3}\mathring{L}^3\\
	&\quad + \frac{1}{6}HR_h - \frac{1}{27}H^3 - \frac{1}{3}\Delta_hH,
\end{align*}
and observed that
\begin{equation*}
	4\pi^2\chi(X) = \int_X\left( \frac{1}{8}|W|^2 + \frac{1}{2}Q_g \right)
	dV_g + \oint_M (\mathcal{L} + T)dV_h.
\end{equation*}
Here, $h$ is the induced metric on $M$, $\mu$ is the inward unit normal, $L$ is the second fundamental form and $\mathring{L}$ its tracefree part, and $H = \tr_hL$ is the mean curvature.
Like $Q$, the $T$-curvature transforms linearly under conformal transformation: $\widetilde{T} = e^{-3\omega}(T + P_3\omega)$, where
\begin{align*}
	P_3u &= \frac{1}{2}\mu(\Delta_gu) + \Delta_h\mu(u) - \frac{1}{3}H
	\Delta_hu + \mathring{L}^{\mu\nu}\nabla_{\mu}^h\nabla_{\nu}^hu
	+ \frac{1}{3}H^{\mu}u_{\mu}\\
	&\quad + \left( \frac{1}{6}R_g - \frac{1}{2}R_h
	- \frac{1}{2}|\mathring{L}|_h^2 + \frac{1}{3}H^2\right)
	\mu(u)
\end{align*}
is conformally invariant: $\widetilde{P}_3 = e^{-3\omega}P_3$. In light of this formulation of the Gauss--Bonnet formula, it is natural to ask the Cherrier--Escobar-type 
question \cite{cher84,esc92b}: 
can one make a conformal change to achieve $\widetilde{Q} = 0$ and $\widetilde{T} = const$? The answer, again, is generically yes~\cite{n09}. This may be interpreted as sending as much as possible of the interior
Gauss--Bonnet integral to the boundary; in the case of a locally conformally flat manifold, all of it is sent to the boundary. Because the PDE involved is fourth-order, two boundary conditions are actually needed;
Ndiaye~\cite{n09} imposes, in addition to $\widetilde{T} = const$, that $\widetilde{H} = 0$.

Our paper provides the first step toward solving the same problem for four-manifolds with \emph{corners} of codimension two. Suppose $X$ is a four-manifold with two boundary hypersurfaces $M$ and $N$ that
intersect along $\Sigma^2 = M \cap N$.
McKeown~\cite{m21} showed that there exist curvature quantities $G, U \in C^{\infty}(\Sigma)$ (defined in the next paragraph) and a second-order linear operator 
$P_2:C^{\infty}(X) \to C^{\infty}(\Sigma)$ such that
\begin{equation*}
	4\pi^2\chi(X^4) = \int_{X}\left( \frac{1}{8}|W|^2 + \frac{1}{2}Q_g \right)dV_g + \int_{M \cup N}(\mathcal{L} + T)dV_h + \oint_{\Sigma}(G + U)dV_k,
\end{equation*}
with $G$ a pointwise conformal invariant of weight $-2$ (meaning $\widetilde{G} = e^{-2\omega}G$), 
$U$ satisfying $\widetilde{U} = e^{-2\omega}(U + P_2\omega)$, and $\widetilde{P}_2 = e^{-2\omega}P_2$. We wish
to find $\omega$ such that $\widetilde{Q} = 0$, such that $\widetilde{T} = 0$ on both $M$ and $N$, and such that $\widetilde{U} = const$. Again this can be viewed as sending topological information ``to the corner.''

To define the quantities involved in the above formula, we first let $\theta_0 \in C^{\infty}(\Sigma)$ be the angle, at each point of $\Sigma$, between $M$ and $N$. We define
$k := g|_{T\Sigma}$. Viewing $\Sigma$ as a hypersurface in $M$, we let $II_M$ be its second fundamental form and $\eta_M = \tr_kII_M$ its mean curvature; similarly for $II_N$ and $\eta_N$. Let $\mu_M$ and $\mu_N$ be
the inward unit normals to $M$, $N$, while $\nu_M,\nu_N$ are the inward unit normals to $\Sigma$ in $M$, $N$, respectively. Finally, let $K$ be the Gaussian
curvature of $\Sigma$. Then
\begin{equation*}
	G := \frac{1}{2}\cot(\theta_0)(|\mathring{II}_M|_k^2 + |\mathring{II}_N|_k^2) - \csc(\theta_0)\langle \mathring{II}_M,\mathring{II}_N\rangle_k,
\end{equation*}
\begin{multline*}
	U := (\pi - \theta_0)K - \frac{1}{4}\cot(\theta_0)(\eta_M^2 + \eta_N^2) \\
	 + \frac{1}{2}\csc(\theta_0)\eta_M\eta_N - \frac{1}{3}(\nu_MH_M + \nu_NH_N),
\end{multline*}
and
\begin{equation}
	\label{p2eq}
	\begin{aligned}
	P_2u & := (\theta_0 - \pi)\Delta_ku + \nu_M\mu_Mu + \nu_N\mu_Nu\\
	&+ \cot(\theta_0)(\eta_M\nu_Mu + \eta_N\nu_Nu) - \csc(\theta_0)(\eta_N\nu_Mu + \eta_M\nu_Nu)\\
	&+ \frac{1}{3}(H_M\nu_Mu + H_N\nu_Nu).
	\end{aligned}
\end{equation}
The boundary value problem we wish to study is therefore
\begin{equation}
	\label{eqsys}
	\begin{cases}
	P_4\omega = -Q_g , & \text{in $X$}, \\
	P_3^M\omega = -T_M , & \text{on $M$}, \\
	P_3^N\omega = -T_N , & \text{on $N$}, \\
	P_2\omega = -U + Ce^{2\omega}, & \text{along $\Sigma$} ,
	\end{cases}
\end{equation}
with $C$ determined by Gauss--Bonnet and with some additional first-order conditions to make the problem well-posed.
There are several unusual features about the boundary value problem~\eqref{eqsys} which motivate us in this paper to consider this problem on the half-ball before we pursue our later goal of studying the general
problem:

First, \eqref{p2eq} is a complicated operator, and it is not completely clear whether we should expect the boundary value problem~\eqref{eqsys} to have a solution at all.

Second, \eqref{eqsys} is a rather strange
boundary value problem. The literature on elliptic boundary value problems on cornered spaces is vast, but the typical setup is for an operator of order $2m$ to be prescribed on the interior, and $m$ boundary operators to be
prescribed on the codimension-one boundary components. Here, on the other hand, we have a problem of order four with one boundary condition (so far) on the three-dimensional boundaries, and another prescribed on the
two-dimensional corner. It seems intuitively clear that if no further data are prescribed on the boundaries, the problem will be badly underdetermined. On the other hand, we cannot expect to be able to prescribe
an additional condition arbitrarily on $M$ and $N$: leaving aside questions of regularity at the corner (which can generally be controlled in appropriately weighted spaces), we would generally expect to find an essentially
unique solution given two boundary conditions on each boundary component; and this leaves no freedom to prescribe the $\widetilde{U} = C$ condition that is the goal of the problem.
The freedom to prescribe a second boundary
condition must therefore be somehow limited.
Our secondary goal is to better understand the conditions making (\ref{eqsys}) into a well-posed problem.

Working on the half-ball greatly simplifies the problem and distills its essential features. 
First, we are able to exploit the maximal symmetry of the space. Next, both boundary components and the corner are umbilic, significantly simplifying the equations.

Let $X = B^4_+$, the upper half-ball in $\mathbb{R}^4$. Let $M = S^3_+$, the round part of its boundary, and $N = B^3$, the three-ball, which is the flat part of the boundary. Thus, $\partial X = M \cup N$, and
$\Sigma = M \cap N$ is $S^2$. Now, the \emph{whole} ball $B^4$ has Euler characteristic $1$ and is flat with umbilic boundary, so the only contribution in the Gauss--Bonnet formula is $\oint_{S^3}T dA$.
Since $\vol(S^3) = 2\pi^2$, we conclude that $T \equiv 2$ on $S^3$ (thus, on $M$). Since the half-ball is missing half of the sphere and $B^3$ is both flat and totally geodesic, we conclude that the contribution
of $\oint_{S^2}U$ to the Gauss--Bonnet integrand for $B^4_+$ is $2\pi^2$, from which we conclude $U \equiv \frac{\pi}{2}$. We simplify our task by first looking for solutions $\omega$ that are constant
on $\Sigma = S^2$. In this case, $e^{-2\omega}$ is itself a constant there, and our corner condition reduces to $P_2\omega = \frac{\pi}{2}$ (that is, we want to double $e^{2\omega}U$). Thus, on the half-ball with our
ansatz that $\omega = const$ on $S^2$, the conditions we wish to satisfy are
\begin{equation}
	\label{sphereconds}
	\begin{cases}
		\Delta^2\omega = 0, & \text{in $B_+^4$}, \\
		\frac{1}{2}\mu_M(\Delta_{\mathbb{R}^4}\omega) + \Delta_{S^3}\mu_M(\omega) - \Delta_{S^3}\omega = -2, & \text{on $S_+^3$}, \\
		\frac{1}{2}\mu_N(\Delta_{\mathbb{R}^4}\omega) + \Delta_{\mathbb{R}^3}\mu_N(\omega) = 0, & \text{on $B^3$}, \\
		\nu_M(\mu_M(\omega)) + \nu_N(\mu_N(\omega)) - \nu_M(\omega) = \frac{\pi}{2}, & \text{along $S^2$} .
	\end{cases}
\end{equation}
(Here the Laplace term in $P_2$ vanishes because we assume $\omega$ is constant on the corner.)

What final boundary conditions should we prescribe on $M$ and $N$? One might hope---especially with so much symmetry---that we might prescribe $H = 0$ on both surfaces, analogous to the case of (uncornered) manifolds with boundary~\cite{n09}.
However, this is impossible: both $M$ and $N$ are umbilic, and any conformal change that makes them minimal will therefore make them totally geodesic. In that case, their intersection $\Sigma$ will
also be totally geodesic, and all the boundary and corner contributions to the Gauss--Bonnet formula (except the Gaussian term in $U$, which provides only half of the needed contribution) will therefore vanish.
By Gauss--Bonnet, therefore, no such function can be biharmonic.

In fact, it is not hard to derive a constraint on possible mean curvatures of the transformed metric. Recall that on either boundary, $H$ transforms by 
\begin{equation}
	\label{Htranseq}
	\widetilde{H} = e^{-\omega}(H - 3\mu(\omega)).
\end{equation}
Writing the last
equation of (\ref{sphereconds}) in spherical coordinates yields simply
\begin{equation*}
	2\partial^2_{\rho\phi}\omega = \frac{\pi}{2},
\end{equation*}
where we used $\mu_M = -\frac{\partial}{\partial \rho}$, $\nu_M = -\frac{\partial}{\partial \phi}$, $\mu_N = -\rho^{-1}\frac{\partial}{\partial \phi}$, and $\nu_N = -\frac{\partial}{\partial \rho}$. This, in turn, implies
that we have, separately,
\begin{equation}
	\label{sphereconsts}
	\begin{aligned}
		\nu_M(\mu_M(\omega)) &= \frac{\pi}{4}, \\
		\nu_N(\mu_N(\omega)) - \mu_N(\omega) &= \frac{\pi}{4}.
	\end{aligned}
\end{equation}
Now, taking equation (\ref{Htranseq}) for $M$, applying $\nu_M$, using the fact that $\nu_M = \mu_N$ at the corner, and then repeating for $N$ gives the following equations that must hold along $\Sigma$:
\begin{equation*}
	\begin{aligned}
		\nu_M(\widetilde{H}_M) &= -\frac{3\pi}{4}e^{-\omega} + \frac{1}{3}e^{\omega}\widetilde{H}_N\widetilde{H}_M,\\
		\nu_N(\widetilde{H}_N) &= -\frac{3\pi}{4}e^{-\omega} + \frac{1}{3}e^{\omega}\widetilde{H}_N\widetilde{H}_M.
	\end{aligned}
\end{equation*}
So any possible prescribed mean curvatures will have to satisfy these equations, and in particular, there is a strong constraint on what constants could appear as mean curvatures of $M$ and $N$.

Actually, since our goal is just to study the structure of the boundary conditions and how to make the problem well-posed, we leave aside the question of prescribing mean curvature, and content ourselves with the linear problem: we will prescribe $\mu_M(\omega)$ and $\mu_N(\omega)$. Then (\ref{sphereconsts}) dictates constraints on the prescribed data,
and the clear question is whether satisfying these constraints suffices to guarantee existence.
The answer is yes. Thus, the effect of the corner condition on our freedom is simply the existence of a scalar constraint at the corner on the second boundary
condition for each face.
Our main result is the following.

\begin{Theorem}
    \label{main-thm}
	Let $M = S^3_+$ and $N = B^3$, with $\Sigma = S^2 = M \cap N$. Let $\psi \in C^{\infty}(M)$ and $\varphi \in C^{\infty}(N)$ satisfy
	\begin{equation*}
		\begin{aligned}
			\nu_M(\psi) &= \frac{\pi}{4} , \\
			\nu_N(\varphi) - \varphi &= \frac{\pi}{4}.
		\end{aligned}
	\end{equation*}
	Then there exists $\omega \in C^{3}(B^4_+)$ such that $\omega|_{\Sigma}$ is constant, $\omega$ solves the boundary value problem~\eqref{sphereconds}, and
	\begin{align*}
		\mu_M(\omega) &= \psi, \\
		\mu_N(\omega) &= \varphi.
	\end{align*}
	With the additional condition
	\begin{equation*}
		\omega|_{\Sigma} = 0,
	\end{equation*}
	the solution is unique.
\end{Theorem}
As the proof will make clear, it would suffice to take $\psi, \varphi \in C^{2,\alpha}$. If they are smooth, the solution will be smooth up to the boundary except at the corner, but at the corner, it will generally
not be even $C^4$.
This is a manifestation of the well-studied failure of elliptic regularity near corners~\cite{gri11,np94}.

The proof is fairly elementary in essence. Most of the technicalities arise because we need a solution that is $C^3$ in order for the boundary operators to be well-defined, and yet the solution is not $C^4$ in general.
Thus, we need to attain close-to-optimal regularity. 

Theorem \ref{main-thm} sheds light on how the condition at the corner interacts with the freedom to prescribe a second boundary condition on the boundary hypersurfaces. 
Our ultimate goal is to study the existence 
problem on a general four-manifold with corners. We expect that the insights gained from the current paper regarding constraints will prove very helpful in attacking that question.

In Section~\ref{setsec}, we define our terms and our coordinates, and introduce some basic biharmonic functions, defined in terms of spherical harmonics,
from which we will build the solutions. Section~\ref{sphersec} contains some existence and convergence theorems on the sphere $S^3$ and is where the bulk of our analysis occurs.
We construct the solutions in Section~\ref{conssec}.

All our solutions have $\omega$ constant at the corner. In Section~\ref{confsec}, we consider the action of the conformal group of $B^4_+$ and use it to construct solutions of the boundary value problem~\eqref{eqsys} which are not constant along $\Sigma$.
In particular, the uniqueness of Theorem~\ref{main-thm} fails much more profoundly if $\omega\rvert_\Sigma$ is allowed to be nonconstant.

%% file: tex/set.tex
\section{Setup}\label{setsec}

In this section, we fix our notation and introduce the functions from which
we will construct our solution.

We denote
\begin{equation*}
 X := B^4_+ = \left\{ (x,y,z,w) \suchthat x^2 + y^2 + z^2 + w^2 \leq 1,
w \geq 0\right\},
\end{equation*}
and denote by
\begin{align*}
 M & := S^3_+ = \left\{(x,y,z,w) \suchthat x^2 + y^2 + z^2 + w^2 = 1, w \geq 0  \right\} , \\
 N & := B^3 = \left\{ (x,y,z,0) \suchthat x^2 + y^2 + z^2 \leq 1 \right\} ,
\end{align*}
the boundary pieces separated by the corner $\Sigma := M \cap N$.

We denote by $g = g_E$ the Euclidean metric on $B^4$, which of course
satisfies $Q = 0$.
The half-sphere $M$ satisfies $T_M = 2$ and $H_M = 3$;
it is umbilic, so $\mathring{L}_M = 0$.
The flat boundary $N$ satisfies $T_N = H_N = 0$ and $\mathring{L}_N = 0$. 

The corner is simply the two-sphere of radius one, and so has Gauss curvature $K = 1$.
Viewed as a submanifold of $M$, it is a totally geodesic
equatorial sphere, so $\eta_M = 0$ and $\mathring{II}_M = 0$.
Viewed as a submanifold of $B^3$, it is the two-sphere in three-space,
so $\eta_N = 2$ and $\mathring{II}_N = 0$.

We will work primarily in spherical coordinates $(\rho,\phi,\alpha,\theta)$ on $\mathbb{R}^4$, so that the metric is
\begin{equation*}
 g = d\rho^2 + \rho^2(d\phi^2 + \sin^2(\phi)(d\alpha^2 + \sin^2(\alpha)
d\theta^2)) .
\end{equation*}
In particular, $\phi$ is the polar angle and $\Sigma$
is given by $\rho=1, \phi = \frac{\pi}{2}$.
In these coordinates, the inward unit normals to $M$ and $N$ are given by
$\mu_M = -\frac{\partial}{\partial \rho}$ and
(away from the origin) $\mu_N = -\rho^{-1}\frac{\partial}{\partial \phi}$.
The inward unit normals to $\Sigma$ in $M$ and $N$ are given by
$\nu_M = -\frac{\partial}{\partial \phi} = \mu_N|_{\Sigma}$
and $\nu_N = -\frac{\partial}{\partial \rho} = \mu_M|_{\Sigma}$.

In light of the above computations of mean curvatures, the Chang--Qing operator $P_3$ on $M$ is
given by
\begin{equation*}
	P_3^Mf = \frac{1}{2}\mu_M\Delta_{\mathbb{R}^4}f + \Delta_M(\mu_Mf) - \Delta_M(f),
\end{equation*}
while that on $N$ is given by
\begin{equation*}
	P_3^Nf = \frac{1}{2}\mu_N\Delta_{\mathbb{R}^4}f + \Delta_N(\mu_Nf).
\end{equation*}
Similarly, the $P_2$ operator on $\Sigma$ is given by
\begin{equation*}
	P_2f = -\frac{\pi}{2}\Delta_{S^2}f + \nu_M\mu_Mf + \nu_N\mu_Nf - \nu_Mf.
\end{equation*}

We will introduce two special families of biharmonic functions on the ball.
Before we do so, we briefly review the spherical harmonics on $S^3$. There are many sources for the following information;
we follow Higuchi~\cite{h87}, though there are also good textbook references~\cite{mul98,sau06,sw71}.
For each $k \in \mathbb{N} \cup \left\{ 0 \right\}$, the spherical
Laplacian $\Delta_{S^3}$ has an eigenvalue $-k(k + 2)$ with a $(k+1)^2$-dimensional eigenspace. A natural basis is given
by the \emph{spherical harmonics}, which are parametrized by $k \in \mathbb{N} \cup \left\{ 0 \right\}$, $l \in \left\{ -k,1-k,\dotsc,k - 1,k \right\}$ and
$p \in \left\{ |l|,\dotsc,k \right\}$. For each such choice, we may define a smooth function $\hat{f}_{k,p,l} \in C^{\infty}(S^3)$;
collectively, these form an orthonormal basis for $L^2(S^3)$.

With respect to the natural involution $\tau:S^3 \to S^3$,
\begin{equation*}
 \tau(x,y,z,w) := (x,y,z,-w) ,
\end{equation*}
inverting the three-sphere about its equatorial sphere, the spherical harmonics with $k - p$ even (resp.\ $k - p$ odd) are even (resp.\ odd).
This follows from \cite[equation (2.8)]{h87} and \cite[equation (7.3.1.86)]{pbm86iv}.
We are interested in the \emph{half}-sphere $M = S^3_+$, and on $M$, the even and odd spherical harmonics each give a separate
orthogonal basis for $L^2(M)$.
To see this, simply observe that any $L^2$ function $f$ on $M$ can be extended to $S^3$ as an even (resp.\ odd) function.
The resulting function on $S^3$ can then be expanded in spherical harmonics, and the expansion will consist entirely of even (resp.\ odd)
functions.
Since either expansion restricts to the half-sphere as an expansion of $f$, we see that the even and odd harmonics each gives an orthogonal basis for the
half-sphere. It is clear that each set remains mutually orthogonal, although an even and an odd spherical harmonic will not be orthogonal to each other on the half-sphere.

In order to have ortho\emph{normal} bases for $L^2(M)$, we must multiply each spherical harmonic by $\sqrt{2}$. We thus define $f_{k,p,l} := \sqrt{2}\hat{f}_{k,p,l}$.
We will primarily be interested in the zonal harmonic~\cite{sw71};
i.e.\ the spherical harmonic $f_{k,0,0}$ which is independent of $\alpha$ and $\theta$.
Since it is unambiguous,
we will refer to these as $f_k$.
These normalized zonal harmonics are given by
\begin{equation*}
	f_k(\phi) = \frac{\sin((k + 1)\phi)}{\pi\sin(\phi)}.
\end{equation*}

We require two infinite families of biharmonic functions on the ball:
\begin{align*}
	F_{k,p,l,1}(\rho,\phi,\alpha,\theta) & := (k + 2 - k\rho^2)\rho^kf_{k,p,l}(\phi,\alpha,\theta) ,\\
	F_{k,p,l,2}(\rho,\phi,\alpha,\theta) & := (\rho^2 - 1)\rho^kf_{k,p,l}(\phi,\alpha,\theta) .
\end{align*}
It is straightforward to compute that these are indeed biharmonic.
(In keeping with the above convention, and to minimize notational clutter,
we will set $F_{k,1} := F_{k,0,0,1}$ and $F_{k,2} := F_{k,0,0,2}$.)

Although it is not required for our construction, it may be instructive to note how $F_{k,p,l,1}$ and $F_{k,p,l,2}$ were derived.
The bilaplacian $\Delta^2$ on the Euclidean ball is, in particular, the Paneitz operator corresponding to the metric $g_E$. The Paneitz operator $P_4$ is a linear fourth-order operator
with principal part $\Delta^2$ which, under the conformal transformation $\tilde{g} = e^{2\omega}g$, satisfies $\widetilde{P}_4 = e^{-4\omega}P_4$. Thus biharmonic functions are also in the
kernel of the Paneitz operator $P_4^+$ of the hyperbolic metric $g_+ = \frac{4}{(1 - \rho^2)^2}g_E$. But this operator factors as $P_4^+ = \Delta_+(\Delta_+ + 2)$, where $\Delta_+$ is the
Laplace-Beltrami operator on hyperbolic space. It is trivial that each factor restricts to an isomorphism on the kernel of the other, and it follows that
$\ker(P_4) = \ker(P_4^+) = \ker(\Delta_+) \oplus \ker(\Delta_+ + 2)$. These second-order operators can be easily solved using separation of variables. Each yields two infinite families of
solutions, half of which are unbounded at the origin. The other two are above.

In light of the fact that the even and odd spherical harmonics are independently a basis of $L^2(M)$, we in fact find it useful to regard the above two families of functions as
\emph{four} families of functions, depending on the parity of $k$. The heart of our construction is Table~\ref{maintab}, which shows how each of the boundary operators
we care about acts on each of our zonal biharmonic functions.
\begin{table}
	\label{maintab}
	\caption{Behavior of basic functions under boundary operators}
\begin{tabular}{|c||c|c|c|c|}
	\hline
	& $P_3^M$ & $\mu_M$ & $P_3^N$ & $\mu_N$\\\hline
	$F_{2j,1}$ & $8j(j+1)(2j+1)f_{2j}$ & $0$ & $0$ & $0$\\\hline
	$F_{2j + 1,1}$& $4(j+1)(2j+1)(2j+3)f_{2j+1}$ & $0$ & $A_{2j+1,1}$ & $B_{2j+1,1}$\\\hline
	$F_{2j,2}$ & $0$ & $-2f_{2j}$ & $0$ & $0$\\\hline
	$F_{2j+1,2}$ & $0$ & $-2f_{2j+1}$ & $A_{2j + 1,2}$ & $B_{2j+1,2}$.\\\hline
\end{tabular}
\end{table}
In this table, $A$ and $B$ are polynomials in $\rho$ on $N$.
To derive the table, recall that
$\Delta = \partial_{\rho}^2 + \frac{3}{\rho}\partial_{\rho} + \rho^{-2}\Delta_{S^3}$.
We compute that
\begin{align*}
	\Delta F_{k,1} &= -4k(k+2)\rho^kf_{k} ,\\
	\mu_M F_{k,1} &= 0 ,\\
	B_{k,1} = \mu_N F_{k,1} &= -(k + 2 - k\rho^2)\rho^{k - 1}f_{k}'\left( \frac{\pi}{2} \right) ,\\
	\Delta_{\mathbb{R}^3}\mu_N F_{k,1} &= -k(k + 2)((k - 1) - (k + 1)\rho^2)\rho^{k - 3}f_{k}'\left( \frac{\pi}{2} \right) , \\
	A_{k,1} = P_{3}^NF_{k,1} &= -k(k+2)(k - 1 - (k+3)\rho^2)\rho^{k - 3}f_k'\left( \frac{\pi}{2} \right),\\
	\Delta F_{k,2} &= 4(k+2)\rho^kf_{k} , \\
	\mu_M F_{k,2} &= -2f_{k} , \\
	B_{k,2} = \mu_N F_{k,2} &= -(\rho^2 - 1)\rho^{k - 1}f_{k}'\left( \frac{\pi}{2} \right) , \\
	\Delta_{\mathbb{R}^3}\mu_N F_{k,2} &= (k(k-1)-(k+1)(k+2)\rho^2)\rho^{k - 3}f_{k}'\left( \frac{\pi}{2} \right) , \\
	A_{k,2} = P_{3}^NF_{k,2} &= (k(k-1)-(k+2)(k+3)\rho^2)\rho^{k - 3}f_k'\left( \frac{\pi}{2} \right).
\end{align*}

%% file: tex/sphere.tex
\section{Analysis on the sphere}\label{sphersec}

In this section we carry out some analysis on $S^3$ needed for our construction of solutions to the boundary value problem~\eqref{sphereconds}.

First, we state a theorem of Schechter~\cite{sch63}.
\begin{theorem}
	\label{schecthm}
	Suppose $A$ is an elliptic operator of order $m \geq 2$ and that $\left\{ B_j \right\}_{j = 1}^{m/2}$ is a system of boundary operators which, together with $A$, satisfy the Lopatinskii--Shapiro conditions on a smoothly
	bounded domain $\Omega$. We let $m_j$ be the order of $B_j$. Let $p \geq 2$. Then for all real $s$, there is a constant $C > 0$ such that for all $u \in C^{\infty}(\overline{\Omega})$,
	\begin{equation*}
	 \lVert u \rVert_{s,p} \leq C(\lVert Au \rVert_{s - m,p} + \sum_{j = 1}^{m/2}\lVert B_ju \rVert_{s - m_j-1/p,p}^{\partial \Omega} + \lVert u \rVert_{s - m,p}).
	\end{equation*}
\end{theorem}

For positive integers $s$, the norm $\lVert \cdot \rVert_{s,p}$ is the usual one on the Sobolev space $H^{s,p}$.
For negative integers $s$, it is defined by duality.
For real $s$, it is defined by complex interpolation~\cite{ada75,c64,l60}.
The boundary norms are the same, defined by partition of unity and charts.

\begin{proposition}
	\label{smoothprop}
	Suppose $\varphi \in C^{\infty}(S^3)$. Then there exists $u \in C^{\infty}(B^4)$ such that
	\begin{equation*}
	 \begin{cases}
		\Delta^2u = 0 , &\text{in $B^4$} ,\\
		P_3^{S^3}u = 0 , &\text{on $S^3$} ,\\
		\mu_{S^3}u = \varphi , &\text{on $S^3$} .
	 \end{cases}
	\end{equation*}
	Moreover, $u$ is constant on $S^3$ and, subject to the condition that this constant is zero, is unique. Finally, for all such $u$ vanishing
	on $S^3$, 
	\begin{equation}
		\label{smoothest}
		\lVert u \rVert_{L^{\infty}(B^4)} \leq \lVert\varphi\rVert_{L^{\infty}(S^3)} .
	\end{equation}
\end{proposition}

\begin{proof}
	Since $\varphi \in L^2(S^3)$, we may write $\varphi = \sum_{k = 0}^{\infty}\sum_{m = 1}^{(k + 1)^2}c_{k,m}f_{k,m}$, where $\left\{ f_{k,m} \right\}$ are the spherical harmonics of order $k$. (For convenience, 
	we here compress the indices $p$ and $l$ into a single index $m$.) This series converges in $L^2$ and, because $\varphi$ is smooth, also converges uniformly~\cite{k95}.
	We let $\varphi_N = \sum_{k = 0}^N\sum_{m = 1}^{(k + 1)^2}c_{k,m}f_{k,m}$ be the partial sums (over $k$) of this series.

	Since $\varphi$ is smooth, $\varphi \in H^{s,p}$ for any $s$ and for all $p > 1$.
	Due to an equivalent characterization of
	the spaces $H^{s,2}$ by Lions and Magenes~\cite[Remark 7.6]{lm12i}, we deduce that $\sum_{k,m}k^{2s}|c_{k,m}|^2 < \infty$.
	Therefore $\varphi_N$ converges to $\varphi$ in $H^{s,2}$. Consequently, the Sobolev embedding
	theorem~\cite[Prop.\ 4.3.3]{tay11} implies $\varphi_N$ converges to $\varphi$ in $C^{k,\alpha}(S^3)$ for any $k,\alpha$.

	We now define
	\begin{equation*}
		u_N := -\frac{1}{2}\sum_{k = 0}^N\sum_{m=1}^{(k+1)^2}c_{k,m}(\rho^2 - 1)\rho^kf_{k,m}
	\end{equation*}
	on $B^4$, and
	\begin{equation*}
	 u := \lim_{N \to \infty}u_N = -\frac{1}{2}\sum_{k = 0}^{\infty}\sum_{m = 1}^{(k + 1)^2}c_{k,m}(\rho^2 - 1)\rho^kf_{k,m} .
	\end{equation*}
	Because the series expansion of $\varphi$ converges uniformly and $\rho^k$ is bounded and semi-monotonic decreasing, it follows from Abel's test~\cite[Section~III.19]{brom08} that the series defining
	$u$ converges uniformly to a (therefore) continuous function on $B^4$. Meanwhile, by straightforward term-wise computation (see Table~\ref{maintab}), each $u_N$ satisfies
	\begin{align*}
		\Delta_{\mathbb{R}^4}^2u_N &= 0 ,\\
		P_3^{S^3}u_N &= 0 , \\
		\mu_{S^3}u_N &= \varphi_N.
	\end{align*}
	Each $u_N$ is also smooth, since $\rho^kf_{k,m}$ is a harmonic polynomial.
	Moreover, $u_N \rvert_{S^3} = 0$, and hence $u \rvert_{S^3} = 0$.

	Case~\cite{c18} observed that the boundary value problem with operator $\Delta^2$ and boundary operators $P_3$ and $\mu$ is elliptic --- that is, it satisfies the Lopatinskii--Shapiro conditions. We may thus
	apply Theorem~\ref{schecthm} to conclude that
	\begin{equation*}
		\lVert u_N - u_L \rVert_{4,2} \leq C(\lVert\varphi_N - \varphi_L\rVert_{5/2,2} + \lVert u_N - u_L\rVert_{0,2})
	\end{equation*}
	for all $L,N \in \mathbb{N}$.
	The right-hand side goes to zero for $L,N$ large, so we conclude that $\left\{ u_N \right\}$ is Cauchy, and hence convergent, in $H^{4,2}$. We may iterate this argument to conclude that $u$ is smooth
	and satisfies the desired boundary conditions.

	We now turn to uniqueness. Suppose $u$ satisfies the equations with $\varphi = 0$. Since $\mu_{S^3}(u) = 0$, the condition $P_3u = 0$ reduces to $\mu_{S^3}\Delta_{\mathbb{R}^4}u = 2\Delta_{S^3}u$. We conclude
	\begin{align*}
		0 &= \int_{B^4}u\Delta_{\mathbb{R}^4}^2udx\\
		&= \int_{B^4}(\Delta_{\mathbb{R}^4} u)^2dx + \int_{S^3}\left( \mu(u)\Delta_{\mathbb{R}^4}u - u\mu\Delta_{\mathbb{R}^4}u \right)d\sigma\\
		&= \int_{B^4} (\Delta_{\mathbb{R}^4}u)^2 dx - 2\int_{S^3}u\Delta_{S^3}u d\sigma\\
		&= \int_{B^4}(\Delta_{\mathbb{R}^4}u)^2 dx + 2\int_{S^3}\lvert du \rvert^2d\sigma.
	\end{align*}
	It follows that $u$ is constant on $S^3$ and is harmonic in $B^4$. Thus, any two solutions to the inhomogeneous problem differ by a constant.

	Only the last claim remains. Let 
	\[v_N = \sum_{k = 0}^N\sum_{m=1}^{(k+1)^2}c_{k,m}\rho^kf_{k,m} = -2(\rho^2 - 1)^{-1}u_N.\]
	Each $v_N$ is harmonic, and $v_N|_{S^3} = \varphi_N$. Thus, $\left\{ v_N \right\}$ is a bounded
	sequence of harmonic functions, and so has a convergent subsequence with harmonic limit~\cite[Theorem~2.11]{gt98}. 
	The limit function is clearly $v = \sum_{k = 1}^{\infty}\sum_{m=1}^{(k+1)^2}c_{k,m}\rho^kf_{k,m}$, which takes boundary values
	$\varphi$. By the maximum principle, $\lvert\lvert v\rvert\rvert_{\infty} \leq \lvert\lvert\varphi\rvert\rvert_{\infty}$. The claim follows since $u = -\frac{1}{2}(\rho^2 - 1)v$.
\end{proof}

This proposition enables us to solve the boundary value problem with less regular boundary data.

\begin{theorem}
	\label{spherethm}
	Suppose $\varphi \in C^{2,\alpha}(S^3)$, where $0 < \alpha \leq 1$. Then for any $\beta < \alpha$, there exists $u \in C^{\infty}(\mathring{B}^4) \cap C^{3,\beta}(B^4)$ such that
	\begin{equation*}
	 \begin{cases}
		\Delta^2u = 0 , &\text{in $B^4$} , \\
		P_3^{S^3}u = 0, &\text{on $S^3$} , \\
		\mu_{S^3}u = \varphi, &\text{on $S^3$} .
	 \end{cases}
	\end{equation*}
	It may be chosen so that $u|_{S^3} = 0$, and subject to this choice, it is unique.
\end{theorem}
\begin{proof}
	It follows from a result of Stein~\cite{ste61} that if $\alpha' < \alpha$ and $p \in [2,\infty)$, then $\varphi \in H^{2 + \alpha',p}$.  Choose such an $\alpha'$ in $\left(\beta,\min\left(\alpha,
	\frac{3 + \beta}{4}\right)\right)$, and let
	$p > \frac{3}{\alpha' - \beta}$.
	Because $C^{\infty}(S^3)$ is dense in
	$H^{2+\alpha',p}$, we may find $\varphi_N \in C^{\infty}(S^3)$ such that $\varphi_N \to \varphi$ in $H^{2 + \alpha',p}$.
	The Sobolev embedding theorem then implies that	$\varphi_N \to \varphi$ in $C^{0}(S^3)$.
	By Proposition~\ref{smoothprop}, we may uniquely find $u_N \in C^{\infty}(B^4)$ satisfying $\Delta^2u_N = 0$, $P_3(u_N) = 0$, and $\mu(u_N) = \varphi_N$ with $u_N|_{S^3} = 0$.

	Let $L,N \in \mathbb{N}$. Our choice of $\alpha'$ implies that $-1 + \frac{1}{p} + \alpha' < 0$, so
    	Theorem~\ref{schecthm} implies that
	\begin{equation*}
		\lVert u_N - u_L\rVert_{3 + \alpha' + \frac{1}{p},p} \leq C(\lVert\varphi_N - \varphi_L\rVert_{2 + \alpha',p} + \lVert u_N - u_L\rVert_{0,p}).
	\end{equation*}
	The first term on the right-hand side converges to 0 by hypothesis; the second converges to zero by the uniqueness and Estimate~\eqref{smoothest} of Proposition~\ref{smoothprop}.
	Consequently, $\left\{ u_N \right\}$ is Cauchy in $H^{3 + \alpha' + \frac{1}{p},p}$. Since $p > \frac{3}{\alpha' - \beta}$, we see that $3 + \alpha' + \frac{1}{p} > 3 + \beta + \frac{4}{p}$.
	The Sobolev embedding theorem~\cite[Theorem~7.63 and Section~7.65]{ada75} implies that the sequence is Cauchy in $C^{3,\beta}(B^4)$.
	Let $u$ be the limit function. It is immediate that it satisfies $P_3u = 0$ and $\mu_{S^3}(u) = \varphi$. It also satisfies $\Delta^2u = 0$ (in
	$H^{\alpha' + \frac{1}{p} - 1,p}$); so by local elliptic regularity, $u$ is biharmonic and smooth on the interior.
\end{proof}

We now turn to some regularity analysis of functions expanded entirely in zonal harmonics; i.e.\ functions on the sphere that depend only on the height. We start with the following small technical lemma.

\begin{lemma}
	\label{reglem}
	Suppose that $f \colon \left[ 0,\pi \right] \to \mathbb{R}$ is $C^k$.
	Viewed as a function of the polar angle, $f$ defines a $C^k$ function $\varphi$ on $S^3$ via $\varphi(\phi,\alpha,\theta) = f(\phi)$ if and only if
	its odd derivatives of order less than or equal to $k$ all vanish at the endpoints.
\end{lemma}
\begin{proof}
	Recall that we are using the coordinates $(x,y,z,w)$ on $\mathbb{R}^4$.

	Suppose $f$ defines a $C^k$ function $\varphi$ on the sphere. Since $\phi$ is a smooth coordinate away from the poles, the only thing to discuss is the poles; we focus on $N = (0,0,0,1)$.
	Let $\gamma(t) = (0,0,\sin(t),\cos(t))$; then $\gamma$ is a smooth curve, so $h := \varphi \circ \gamma$ is a $C^k$ function on $\mathbb{R}$. But since $w(\gamma(t)) = w(\gamma(-t))$ and
	$\varphi$ depends only on $w$, we conclude that $h$ is an even function of $t$. Now, $h(t) = f(\lvert t \rvert)$, so considering only non-negative values of $t$, we have
	$h^{(j)}(0) = f^{(j)}(0)$ for all $j$ odd. The claim follows.

	Conversely, suppose that all the odd derivatives of $f$ of order at most $k$ vanish. We assume $k$ is odd; the case of $k$ even is similar but slightly more straightforward. By Taylor's theorem, there is a
	polynomial $p$ of order $\frac{k - 1}{2}$ such that $f(\phi) = p(\phi^2) + \sin^k(\phi)\psi(\phi)$, where $\lim_{\phi \to 0}\psi(\phi) = 0$ and where $\sin^k(\phi)\psi(\phi)$ is $k$-times continuously differentiable
	with $k$-th derivative vanishing at $\phi=0$. But $\phi^2$ is a smooth function on the sphere, being smoothly related to $\sin^2(\phi) = x^2 + y^2 + z^2$; thus $p(\phi^2)$ is smooth on the sphere.
	
	Near the pole, we can take $(x,y,z)$ as a coordinate chart. Writing out
	\begin{align*}
		\frac{\partial}{\partial x} &= \frac{\sin(\alpha)\cos(\theta)}{\cos(\phi)}\frac{\partial}{\partial \phi} + \frac{\cos(\alpha)\cos(\theta)}{\sin(\phi)}\frac{\partial}{\partial \alpha}
		- \frac{\sin(\theta)}{\sin(\phi)\sin(\alpha)}\frac{\partial}{\partial \theta} , \\
		\frac{\partial}{\partial y} &= \frac{\sin(\alpha)\sin(\theta)}{\cos(\phi)}\frac{\partial}{\partial \phi} + \frac{\cos(\alpha)\sin(\theta)}{\sin(\phi)}\frac{\partial}{\partial \alpha}
		+ \frac{\cos(\theta)}{\sin(\phi)\sin(\alpha)}\frac{\partial}{\partial \theta}, \\
	 \frac{\partial}{\partial z} &= \frac{\cos(\alpha)}{\cos(\phi)}\frac{\partial}{\partial \phi} - \frac{\sin(\alpha)}{\sin(\phi)}\frac{\partial}{\partial \alpha} ,
	\end{align*}
	we can see that the application of any $k$ (or fewer) basis derivatives to $\sin^k(\phi)\psi(\phi)$ vanishes at $\phi = 0$. Thus, $\varphi$ is $C^k$.
\end{proof}

We can now prove the following convergence and regularity theorem for functions defined as series of zonal harmonics. The slightly awkward statement is for easy application later and relatively
easy proof. Recall that $N, S$ are the north and south poles of the sphere.

\begin{theorem}
	\label{convthm}
	Let $q \in \mathbb{N} \cup \left\{ 0 \right\}$. Suppose that $\left\{ c_j \right\}_{j = 1}^{\infty}$ is a sequence of strictly positive real numbers that converges to $c > 0$ 
	at least as fast as $c + Cj^{-\varepsilon}$ (some $C, \varepsilon > 0$).
	Let $p_j:\mathbb{R} \to \mathbb{R}$ ($j \in 2\mathbb{Z}$)
	be a sequence of polynomials of the form $r_j(x)x^j$, where $r_j$ is even, of degree bounded in $j$, satisfies $r_j(1) = 1$, 
	and is such that $p_j$ satisfies $\lVert p_j^{(m)} \rVert_{L^{\infty}([0,1])} \leq Cj^m$ for all $m \leq q$.
	Define
	\begin{equation*}
		u(\rho,\phi) := \sum_{j = 1}^{\infty}(-1)^jc_jj^{-(q + 2)}p_{2j}(\rho)f_{2j}(\phi).
	\end{equation*}
	Then $u \in C^{\infty}(\mathring{B}^4) \cap C^q(B^4 \setminus \left\{ N,S \right\})$. In particular, $u$ and its first $q$ term-wise derivatives converge uniformly and absolutely on the complement of any open
	set containing the poles.
	
	If $r_j = 1$ for all $j$, then $u \in C^q(B^4)$, with the series for the highest derivatives converging uniformly but not absolutely near the poles;
	and if the $c_j$ form a monotonic sequence, the partial sums of the $(q+1)$st tangential derivatives of $u$ on the sphere are uniformly bounded on the complement of any open set containing the central slice $B^3$.
\end{theorem}

\begin{proof}
	On the interior, $u$ and all its derivatives are smooth by the Weierstrass $M$-test. Note that $\rho^jf_j(\phi)$ is a harmonic polynomial, and since $r_{2j}$ is even, each term is smooth at the origin.

	Let $\phi_0 > 0$, and consider the set
	\begin{equation*}
	 B_{\phi_0} = \left\{ p \in B^4 \suchthat \phi_0 < \phi(p) < \pi - \phi_0, \frac{1}{2} < \rho \leq 1\right\} .
	\end{equation*}
	Since $\sin(\phi)$ is smooth and uniformly bounded away from zero on this set, we might as well consider
	\begin{equation*}
		\Phi := \pi\sin(\phi)u = \sum_{j = 1}^{\infty}(-1)^jc_jj^{-(q + 2)}p_{2j}(\rho)\sin((2j + 1)\phi).
	\end{equation*}
	Each $\rho$- or $\phi$-derivative increases the $L^{\infty}$-norm of $p_{2j}(\rho)\sin((2j+1)\phi)$ by a factor of order $j$. Since $\sum c_jj^{-2}$ converges absolutely, it is thus immediate that the series,
	with all its derivatives up through order $q$, converges uniformly and absolutely on $B_{\phi_0}$. We have thus shown that $u \in C^{q}(B^4\setminus\left\{ N,S \right\})$. 
	
	We now assume $r_j = 1$ and focus entirely on the north pole $N$; the south pole is equivalent. We will first consider $u$ and its tangential derivatives on the sphere itself, taking $\rho = 1$.

	Since $\phi$ is not a coordinate at $N$, we view $u$ as a function on $[0,1] \times \left[ 0,\frac{\pi}{2} \right]$ and then apply Lemma \ref{reglem}.

	Recall Lagrange's trigonometric identity:
	\begin{equation*}
		\sum_{k = 0}^{n}\cos\left( (2k+1)\theta \right) = \frac{\sin(2(n+1)\theta)}{2\sin(\theta)}.
	\end{equation*}
	Replacing $\theta$ by $\theta - \frac{\pi}{2}$ yields
	\begin{equation*}
		\sum_{k = 0}^n(-1)^k\sin\left( (2k+1)\theta \right) = \frac{(-1)^n\sin(2(n+1)\theta)}{2\cos(\theta)}.
	\end{equation*}
	Finally, we find that
	\begin{align*}
		S_n & := \pi\sum_{k = 0}^n(-1)^kf_{2k}(\phi) \\
		 & = \csc(\phi)\sum_{k = 0}^n(-1)^k\sin\left(( 2k + 1)\phi \right) \\
		 & = \frac{(-1)^n\sin(2(n+1)\phi)}{2\sin(\phi)\cos(\phi)}.
	\end{align*}
	By multiplying together the Laurent series of $\sin(2(n+1)\phi)$, $\sec(\phi)$, and $\csc(\phi)$, we easily see that, for fixed $m \geq 0$ and on a sufficiently small neighborhood $U$ of $\phi = 0$,
	the $m$-th $\phi$-derivative of $S_n$ is $O(n^{m + 1})$ uniformly in $n$ and $\phi$.

	Recall~\cite[Equation (1.2.1)]{zyg68} that, for any sequences $\left\{ a_j \right\}$ and $\left\{ b_j \right\}$, with $A_n = \sum_{j = 1}^na_j$,
	\begin{equation}
		\label{sumbyparts}
		\sum_{j = 1}^na_jb_j = \sum_{j = 1}^{n - 1}A_j(b_j - b_{j + 1}) + A_nb_n.
	\end{equation}
	Let $0 \leq m \leq q$ and let $a_j = (-1)^j\pi \partial_{\phi}^mf_{2j}$, so that $A_j = \partial_{\phi}^mS_j$. Thus, $A_j = O(j^{q + 1})$. Let $b_j = c_jj^{-(q + 2)}$. It follows from our condition on $c_j$ that
	$|b_j - b_{j + 1}| = O(j^{-(q + 2 + \varepsilon)})$. Consequently, the right-hand side of (\ref{sumbyparts}) converges uniformly and absolutely, and the left-hand side converges uniformly.

	Since each $f_{k}$ is an even function of $\phi$
	and we can differentiate term by term, the odd derivatives through order $q$ vanish at $\phi = 0$. We can now apply Lemma \ref{reglem}. This shows that up to $q$ \emph{tangential} derivatives
	of $u$ converge on the sphere.

	Since $u$ is smooth on the inside, tangential (that is, $\phi$) derivatives of $u$ also converge in the interior. Because $r_j = 1$, $u$ is a power series in $\rho$, and we
	can apply Abel's theorem~\cite[pp.\ 128--131]{brom08} to conclude that $\lim_{\rho \to 1^-}u(\rho,\omega) = u(1,\omega)$ for all $\omega \in S^3$. The same argument we have just gone through
	in the last paragraphs works to show continuity up to boundary of the $\rho$ derivative of up to $q - 1$ tangential derivatives.

	For $0 \leq k \leq q$, we can again apply the argument of the last several paragraphs to show that $\partial_{\rho}^ku$ is tangentially $C^{q - k}$, with all mixed derivatives continuous up to the sphere.

	It remains to show that the tangential $(q+1)$st derivatives are uniformly bounded if the $c_j$ are monotonic. We again work on the sphere, and return to the context of (\ref{sumbyparts}). In our setting, the right-hand side
	reads
	\begin{align*}
		&\sum_{j = 1}^{n - 1}\left[\left((c_j - c_{j + 1})(j+1)^{-(q + 2)} + \frac{(q + 2)c_j}{j(j+1)^{q + 2}} + \mathcal{C}\right)\partial_{\phi}^{q + 1}S_j\right]\\
		&\quad+ c_nn^{-(q + 2)}\partial_{\phi}^{q + 1}S_n,
	\end{align*}
	where $\mathcal{C}$ converges to 0 as fast as $j^{-(q + 4)}$.
	The last term is uniformly bounded, but not necessarily convergent as $n$ approaches infinity. As for the series, $\sum_{j = 0}^{\infty}|c_j - c_{j + 1}|$ converges by monotonicity 
	and $(j+1)^{-(q + 2)}\partial_{\phi}^{q + 1}S_j$
	is uniformly bounded, so the first term in the series converges absolutely. But for the second, $\sum_{j = 0}^{\infty}(c_j - c_{j + 1})j^{-1}$ also converges absolutely,
	and so it follows by \cite[Theorem I.2.4]{zyg68} that the series of second terms converges uniformly. The terms involving $\mathcal{C}$ are trivially
	convergent. The extension to the inside of the ball, $\rho < 1$, then follows from another straightforward application of (\ref{sumbyparts}).
\end{proof}

%% file: tex/cons.tex
\section{Construction}\label{conssec}

We now prepare to prove our main result.

First, we define a map $\Lambda:B^4_+ \to B^4_+$ by 
\begin{equation}
 \label{eqn:inversion}
 \Lambda(x,y,z,w) := \frac{(2x,2y,2z,1-x^2-y^2-z^2-w^2)}{x^2 + y^2 + z^2 + (w + 1)^2}.
\end{equation}
Elementary calculations show that $\Lambda$ is a diffeomorphism that interchanges $M$ and $N$; i.e.\ it restricts to diffeomorphisms $\Lambda|_{M}:M \to N$ and $\Lambda|_N: N \to M$. It fixes $\Sigma$ pointwise.
Morever, $\Lambda^2=\mathrm{Id}$.
It is also easy to compute that $\Lambda$ is a conformal transformation:
\begin{equation*}
	\Lambda^*g_E = \Omega^2g_E,
\end{equation*}
where
\begin{equation*}
	\Omega = \frac{2}{1 + 2\rho\cos\phi + \rho^2} = \frac{2}{x^2 + y^2 + z^2 + (w + 1)^2}.
\end{equation*}

\begin{proof}[Proof of Theorem A]
	We first construct $u_0 \in C^3(B^4_+)$ satisfying
	\begin{align*}
		\Delta^2u_0 &= 0\\
		P_3^Mu_0 &= \frac{1}{\pi}.
	\end{align*}
	(We later multiply by $-2\pi$ to achieve the desired boundary condition on $M$.)
	The function $u_0$ will not satisfy $P_3^Nu_0 = 0$, so we will need to perturb it to achieve this at the next step.

	Note that $P_3^Mu_0 = \frac{1}{\pi}$ if and only if $P_3^Mu_0 = f_0$. Looking at the first line of Table \ref{maintab} makes us want to take $u_0 = F_{0,1}$, but unfortunately,
	$P_3^M(F_{0,1}) = 0$. The other even spherical harmonics are all orthogonal to $f_0$, so we must introduce an odd spherical harmonic
	to capture $f_0$. Rather than expanding $f_0$ fully in odd spherical harmonics, it will prove more convenient to expand $f_1$ in the even harmonics and so notice that
	\begin{equation}
		\label{f0exp}
		f_0 = \frac{1}{\langle f_0,f_1\rangle}\left[f_1 - \sum_{j = 1}^{\infty}\langle f_1,f_{2j}\rangle f_{2j}\right].
	\end{equation}
	(Here and after, inner products are with respect to $L^2(S^3_+)$.) This approach will prove to make adjusting $P_3^N$ much easier later.

	Now, recalling that $f_{k}(\phi) = \frac{\sin((k + 1)\phi)}{\pi\sin(\phi)}$ and using integration by parts, we compute that
	\begin{align}
		\langle f_{2k + 1},f_{2j}\rangle &= \frac{8(-1)^{j + k}(k + 1)}{\pi(2k + 2j + 3)(2k - 2j + 1)},
		\intertext{so in particular,}
		\langle f_{1},f_{2j}\rangle &= \frac{8(-1)^{j + 1}}{\pi(2j - 1)(2j + 3)}.
	\end{align}
	Thus, from (\ref{f0exp}) we have
	\begin{equation*}
		f_0 = \frac{3\pi}{8}\left[f_1 + \sum_{j = 1}^{\infty}\frac{8(-1)^{j}}{\pi(2j - 1)(2j + 3)}f_{2j}\right].
	\end{equation*}
	Now, $f_1 = \frac{1}{12}P_3^MF_{1,1}$ and $f_{2j} = \frac{1}{8j(j+1)(2j+1)}P_3^MF_{2j,1}$, so we have
	\begin{multline*}
		f_0 = \frac{3\pi}{8}\Biggl[ \frac{1}{12}P_3^MF_{1,1} \\
		 + \frac{1}{\pi}\sum_{j = 1}^{\infty}\frac{(-1)^j}{j(j+1)(2j-1)(2j+1)(2j+3)}P_3^MF_{2j,1} \Biggr].
	\end{multline*}
	Motivated by this, we formally define
	\begin{align*}
		u_0 & := \frac{3\pi}{8}\left[ \frac{1}{12}F_{1,1} + \frac{1}{\pi}\sum_{j = 1}^{\infty}\frac{(-1)^j}{j(j+1)(2j-1)(2j+1)(2j+3)}F_{2j,1} \right] , 
	\end{align*}
	leaving aside regularity and convergence questions for the moment.
	Note that
	\begin{multline*}
	 u_0 = \frac{3}{8}\Biggl[ \frac{(3-\rho^2)\rho\cos(\phi)}{6} \\
	  + 2\sum_{j = 1}^{\infty}\frac{(-1)^j}{j(j+1)(2j-1)(2j+1)(2j+3)}(j+1-j\rho^2)\rho^{2j}f_{2j} \Biggr] .
	\end{multline*}
	Continuing to work formally, we note (by direct computation or by using Table~\ref{maintab}) that $P_3^N(u_0) = -\frac{3}{4}$, with only the first term contributing.
	Since $P_3^NF_{1,2} = \frac{24}{\pi}$ while $P_3^MF_{1,2} = 0$, we set
	\begin{equation*}
	 u_1 := u_0 + \frac{\pi}{32}F_{1,2} .
	\end{equation*}
    Then
	\begin{multline*}
		u_1 = \frac{1}{8}\Biggl[ \rho\cos\phi \\
		 + 6\sum_{j = 1}^{\infty}\frac{(-1)^j}{j(j+1)(2j-1)(2j+1)(2j+3)}(j + 1-j\rho^2)\rho^{2j}f_{2j} \Biggr].
	\end{multline*}

	We discuss regularity of $u_1$. It follows immediately from Theorem~\ref{convthm} that $u_1$ converges to a $C^3$ function everywhere except possibly the north pole. Near the north pole,
	we can write
	\begin{multline}
		\label{deceq}
		u_1 = \frac{1}{8}\rho\cos\phi + \frac{3}{4}\sum_{j = 1}^{\infty}\frac{(-1)^j}{j(j+1)(2j-1)(2j+1)(2j+3)}\rho^{2j}f_{2j} \\
		\quad\quad+\frac{3}{4}(1 - \rho^2)\sum_{j = 1}^{\infty}\frac{(-1)^j}{(j+1)(2j-1)(2j+1)(2j+3)}
		\rho^{2j}f_{2j}.
	\end{multline}
	The first term is a multiple of $z$. The second term is globally $C^3$ by Theorem~\ref{convthm}. The series in the last term is globally $C^2$, but the sum defining third tangential derivatives
	is bounded near the north pole; thus, due to the factor of $1 - \rho^2$, the last term is globally $C^3$ as well. The convergence in all cases is uniform.

	It now follows that we can apply $\mu$ and $P_3$ term-by-term. We conclude that $P_3^M(u_1) = \frac{1}{\pi}$ and $P_3^N(u_1) = 0$, while
	$\mu_M(u_1) = -\frac{1}{8}\cos(\phi)$ and
	\begin{equation*}
		\mu_N(u_1) = \frac{\pi}{32}(\mu_N(F_{1,1}) + \mu_N(F_{1,2})) = \frac{1}{8}.
	\end{equation*}
	
	We set $\omega_1 := -2\pi u_1$. It is then clear that $P_3^M(\omega_1) = -2$, $P_3^N(\omega_1) = 0$, $\mu_M(\omega_1) = \frac{\pi}{4}\cos(\phi)$ and $\mu_N(\omega_1) = -\frac{\pi}{4}$.

	We turn now to prescribing the normal derivatives. This is easy to do on the three-sphere via Theorem~\ref{spherethm}. As it is much more tedious to do on $B^3$, we proceed in two steps,
	first inverting $B^4_+$ using $\Lambda$ so that we can prescribe data on the sphere instead.

	Let $\hat{g} = \Lambda^*g_E = \Omega^2g_E$ be the pullback of the Euclidean metric by $\Lambda$. We also define $\hat{\omega} := \log\Omega$, so that $\hat{g} = e^{2\hat{\omega}}g_E$. Let
	$\eta = \Lambda|_N:N \to M$. If we parametrize $S^3_+$ by $(\theta,\phi)$ where $\theta \in S^2$ and $\phi \in [0,\frac{\pi}{2}]$, then $\eta^{-1}(\theta,\phi) = (1 + \cos(\phi))^{-1}\sin(\phi)\theta \in B^3$.
	Let $\tilde{\varphi} = \varphi - \mu_N(\omega_1) = \varphi + \frac{\pi}{4}$. Observe that
	\begin{equation*}
		\nu_N(\tilde{\varphi}) - \tilde{\varphi} = 0.
	\end{equation*}
	The product rule implies that $\nu_N(e^{-\hat{\omega}}\tilde{\varphi}) = 0$;
	pulling back by the conformal diffeomorphism gives that $\nu_M\left( (\eta^{-1})^*(e^{-\hat{\omega}}\tilde{\varphi}) \right) = 0$.
	Thus, we can extend $\hat{\varphi} = (\eta^{-1})^*\left( e^{-\hat{\omega}}\tilde{\varphi} \right)$ by reflection to a $C^{2,1}$ function on all of $S^3$. 
	Therefore, by Theorem \ref{spherethm}, there exists a unique $\hat{v}_1 \in C^3(B^4)$ such that
	\begin{align*}
		\Delta^2\hat{v}_1 &= 0 , \\
		P_3^{S_3}\hat{v}_1 &= 0 , \\
		\mu_{S_3}\hat{v}_1 &= \hat{\varphi}.
	\end{align*}
	Since the data $\hat{\varphi}$ is (by construction) invariant under the reflection $w \mapsto -w$, by uniqueness it follows that $\hat{v}_1$ is as well. Consequently, when we restrict $\hat{v}_1$ to $B^4_+$, we also have
	$P_3^N(\hat{v}_1) = 0$ and $\mu_N(\hat{v}_1) = 0$, since both operators vanish on even functions. 
	Now let $v_1 = \Lambda^*\hat{v}_1$. By conformal diffeomorphism, and letting hats on boundary operators indicate they are defined with respect to $\hat{g}$, $v_1$ satisfies
	\begin{align*}
		\Delta_{\hat{g}}^2v_1 &= 0 , \\
		\widehat{P}_3^{M}(v_1) &= 0 , \\
		\widehat{P}_3^{N}(v_1) &= 0 , \\
		\hat{\mu}_M(v_1) &= 0 , \\
		\hat{\mu}_N(v_1) &= \eta^{*}(\eta^{-1})^*e^{-\hat{\omega}}\tilde{\varphi} = e^{-\hat{\omega}}\tilde{\varphi}.
	\end{align*}
	But then, by the conformal transformation laws for $\Delta^2, P_3$, and $\mu$, we have
	\begin{align*}
		\Delta^2v_1 &= 0 , \\
		P_3^M(v_1) &= 0 , \\
		P_3^N(v_1) &= 0 , \\
		\mu_M(v_1) &= 0 , \\
		\mu_N(v_1) &= \tilde{\varphi} = \varphi - \mu_N(\omega_1).
	\end{align*}
	Then setting $\omega_2 = \omega_1 + v_1$, we have $\mu_N(\omega_2) = \varphi$, as desired.

	More straightforwardly, we reflect $\psi - \mu_M(\omega_2)$ again to obtain a $C^{2,1}$ function on $S^3$; this follows since $\nu_M(\psi - \mu_M(\omega_2)) = 0$. Then, again by Theorem~\ref{spherethm},
	we can find $v_2$, biharmonic and satisfying $P_3^M(v_2) = P_3^N(v_2) = 0$, along with $\mu_N(v_2) = 0$ and $\mu_M(v_2) = \psi - \mu_M(\omega_2)$. We set
	\[\omega = \omega_2 + v_2.\]
	This function satisfies all our desired conditions.

	We finally turn to uniqueness. Given two solutions, let $u$ be their difference. Then $u$ satisfies $\Delta^2u = 0$ with the homogeneous boundary conditions, and is $C^3$.
	Because of this and the form of $P_3^N$, we may conclude that $\mu_N(\Delta_{\mathbb{R}^4}u) = 0$. Similarly, on $M$ we conclude that $\mu_M(\Delta_{\mathbb{R}^4}u) = 2\Delta_{M}u$.
	We recall~\cite[p.\ 81]{tay11} that a sufficient condition for the divergence theorem to hold for a given vector field $X$ on a manifold with corners is that $X$ and its divergence both be continuous. Let $X = u\nabla(\Delta u) - (\Delta u)\nabla u$. Then $X$ is continuous, and so is $\diver X = u\Delta^2u - (\Delta u)^2 = -(\Delta u)^2$. Thus, by the divergence theorem,
	\begin{align*}
		-\int_{B^4_+}(\Delta u)^2 dx &= \int_M\left( \Delta_{\mathbb{R}^4}(u)\mu_M(u) - u\mu_M(\Delta_{\mathbb{R}^4}u) \right)dA\\
		&\quad+ \int_N\left( \Delta_{\mathbb{R}^4}(u)\mu_N(u) - u\mu_N(\Delta_{\mathbb{R}^4}u) \right)dA\\
		&= -2\int_{M}u\Delta_{M}udA\\
		&= 2\int_M|\nabla_{S^3} u|^2 dA + 2\int_{\Sigma}u\nu_M(u)d\sigma\\
		&= 2\int_M|\nabla_{S^3} u|^2 dA,
	\end{align*}
	since $u|_{\Sigma} = 0$. We conclude that $\Delta u \equiv 0$. Since $u$ satisfies the homogeneous Neumann condition on both boundaries, we may then apply Green's identity one more time to conclude
	$|\nabla u|^2 \equiv 0$ on $B^4_+$. Thus $u$ is constant, and since it vanishes on $\Sigma$, it vanishes identically.
\end{proof}

\begin{remark}
	We note that the solution is \emph{not} in general $C^4$. In fact, if we take four $\rho$ derivatives of $u_1$ term-by-term (which, on the interior, we may surely do), 
	and observe that $f_{2j}\left( \frac{\pi}{2} \right) =
	\frac{(-1)^j}{\pi}$, we see that at the corner $\rho = 1$, $\phi=\frac{\pi}{2}$, the resulting series is comparable to the positive harmonic series, and so diverges. By using the decomposition~\eqref{deceq} and Abel's theorem~\cite{brom08}, we may conclude that $\partial_{\rho}^4u_1|_{\phi = \frac{\pi}{2}}$ really does approach infinity as $\rho \to 1^-$,
		and this is not a mere artifact of a series representation. 

		On the other hand, it is not hard to show that, away from the corner, $u_1$ is actually smooth up to the boundary. To see this, notice that any neighborhood in $M$ or $N$ \emph{away} from the corner
		can be extended to a compact \emph{smooth} three-manifold contained in $B^4_+$ and bounding a region, say $\Omega$; since $P_3^M(u_1), P_3^N(u_1), \mu_M(u_1)$, and $\mu_N(u_1)$ are all smooth,
		and $u_1$ is smooth in the interior of $B^4_+$, the elliptic boundary value problem $(\Delta^2,P_3^{\partial\Omega},\mu_{\partial \Omega})$ will have $u_1$ as a $C^3$ solution with smooth
		boundary values. Then, elliptic regularty theory in the form of Theorem~\ref{schecthm} will enable us to conclude that $u$ is in fact smooth up to the boundary. That this fails at the corner despite
		smooth boundary data is a manifestation of the general phenomenon whereby elliptic regularity is obstructed at corners.
	\end{remark}

%% file: tex/conf.tex
\section{Action of the conformal group}\label{confsec}

Recall from Section~\ref{setsec} that on the closed Euclidean half-ball $(B_+^4,g)$, the boundary value problem~\eqref{eqsys} simplifies to
\begin{equation}
 \label{eqn:bvp}
 \begin{cases}
  P_4u = 0 , & \text{in $B_+^4$} , \\
  P_3^{S_+^3}u = 2 , & \text{on $S_+^3$} , \\
  P_3^{B^3}u = 0 , & \text{on $B^3$} , \\
  P_2^{S^2}u = \pi e^{2u} - \frac{\pi}{2} , & \text{on $S^2$} .
 \end{cases}
\end{equation}
We constructed solutions to the boundary value problem~\eqref{eqn:bvp} for which $u\rvert_{S^2} = 0$.
In this section we construct additional solutions using the conformal group of $B_+^4$.
To that end, it is convenient to denote by $\mathbb{R}^{1,n}$ the flat Minkowski space $(\mathbb{R}^{n+1},-dt^2+dx_1^2+\dotsm+dx_n^2)$, and by $\TO(1,n)$ the subgroup of the orthogonal group $O(1,n)$ consisting of those elements $\Phi \in O(1,n)$ such that $(t \circ \Phi)(1,0,\dotsc,0)>0$.

The \emph{conformal group} $\Conf(B_+^4)$ is
the group of all diffeomorphisms $\Phi$ of $B_+^4$ which preserve, as sets, each of $S_+^3 \cup B^3$ and $S^2$, and is such that $\Phi^\ast g = e^{2v}g$ for some $v \in C^\infty(B_+^4)$.
Note that $\Conf(B_+^4) \not\subseteq \Conf(B^4)$, due to the existence of elements of $\Conf(B_+^4)$---such as $\Lambda$ defined by Equation~\eqref{eqn:inversion}---which do not preserve $S^3$ as a set;
likewise, $\Conf(B^4) \not\subseteq \Conf(B_+^4)$ due to the existence of elements of $\Conf(B^4)$ which do not preserve $B^3$ as a set.

Our first observation is that $\Conf(B_+^4)$ is noncompact;
indeed, we show that $\Conf(B_+^4) = \Mob(S^2) \rtimes \bZ_2$.
To that end, recall that~\cite{Ratcliffe2006}
\begin{equation*}
 \Mob(S^{n-1}) \cong \Conf(B^n) \cong \TO(1,n) ,
\end{equation*}
where we identify
\begin{equation}
 \label{eqn:flat-ambient-space}
 \begin{split}
 S^{n-1} & \cong \left\{ p = (t,x) \in \mathbb{R} \times \mathbb{R}^{n} \suchthat t^2 = \lvert x \rvert^2 \right\} / ( p \sim \lambda p) , \\
 B^n & \cong \left\{ p = (t,x) \in \mathbb{R} \times \mathbb{R}^n \suchthat t = \sqrt{1 + \lvert x \rvert^2} \right\} .
 \end{split}
\end{equation}
The latter identification is explicitly given via the diffeomorphism
\begin{equation*}
 B^n \ni x \mapsto \left( \frac{1}{(1-\lvert x\rvert^2)^{1/2}} , \frac{x}{(1-\lvert x\rvert^2)^{1/2}} \right) \in \mathbb{R}^{1,n} .
\end{equation*}

We regard $\Mob(S^{n-1})$ as a subgroup of $\Conf(B_+^n)$ as follows:
Given $\Phi \in \Mob(S^{n-1})$, set
\begin{equation}
 \label{eqn:extend-Phi}
 i(\Phi) := \begin{pmatrix} \Phi & 0 \\ 0 & 1 \end{pmatrix} \in \TO(1,n+1) \cong \Conf(B^{n+1}) .
\end{equation}
Since $i(\Phi)$ fixes $x_{n+1}^{-1}(\{0\}) \cong B^n$, we may regard $i(\Phi)$ as an element of $\Conf(B_+^{n+1})$.
We readily check that $i \colon \Mob(S^{n-1}) \to \Conf(B_+^{n+1})$ is an injective group homomorphism.
Moreover, if $\Psi \in \TO(1,n+1)$ has a block diagonal decomposition $\Psi = \begin{pmatrix} \Phi & 0 \\ 0 & 1 \end{pmatrix}$, then necessarily $\Phi \in \Mob(S^{n-1})$.
This allows us to identify $\Mob(S^{n-1})$ as a subgroup of $\Conf(B_+^n)$.
We will abusively use the symbol $\Phi$ to denote both an element $\Phi \in \Mob(S^{n-1})$ and its image under $i$.

\begin{lemma}
 \label{conformal-group}
 We have that $\Conf(B_+^4) = \Mob(S^2) \rtimes \bZ_2$, where $\bZ_2 = \langle \Lambda \rangle$ is the subgroup generated by the conformal map of Equation~\eqref{eqn:inversion}.
\end{lemma}

\begin{proof}
 We must show that the map
 \begin{equation}
  \label{eqn:product}
  \Mob(S^2) \times \bZ_2 \ni (\Phi, n) \mapsto \Phi \Lambda^n \in \Conf(B_+^4)
 \end{equation}
 is bijective and that $\Mob(S^2)$ is a normal subgroup of $\Conf(B_+^4)$.
 
 Denote by
 \begin{equation*}
  \Conf^+(B_+^4) := \left\{ \Psi \in \Conf(B_+^4) \suchthat \Psi(B^3) = B^3 \right\}
 \end{equation*}
 the subgoup of conformal transformations which fix $B^3$ as a set, where we identify
 \begin{align*}
  B^3 & \cong \left\{ (t,x) \in \mathbb{R} \times \mathbb{R}^4 \suchthat t = \sqrt{1 + \lvert x \rvert^2} , x_4=0 \right\} .
 \end{align*}
 Let $\Psi \in \Conf(B_+^4)$.
 We readily compute that $(x_4 \circ \Psi)(p) = 0$ for all $p \in x_4^{-1}(\{0\})$ if and only if $\Psi = \begin{pmatrix} \Phi & 0 \\ 0 & 1 \end{pmatrix}$ for some $\Phi \in O^+(1,3)$.
 Therefore $\Conf^+(B_+^4) = \Mob(S^2)$.
 
 Suppose now that $\Psi \in \Conf(B_+^4) \setminus \Mob(S^2)$.
 Then $\Psi(B^3) = S_+^3$.
 Hence $\Psi\Lambda$ fixes $B^3$ as a set, and so $\Psi \Lambda \in \Mob(S^2)$.
 We conclude that the map~\eqref{eqn:product} is surjective.
 Its injectivity follows easily from the injectivity of $i$.
 
 Finally, let $\Phi \in \Mob(S^2)$ and $\Psi \in \Conf(B_+^4)$.
 By checking separately the cases $\Psi(B^3) = B^3$ and $\Psi(B^3)=S_+^3$, we see that $\Psi^{-1}\Phi\Psi$ fixes $B^3$ as a set.
 Therefore $\Mob(S^2) \unlhd \Conf(B_+^4)$.
\end{proof}

In particular, $\Conf(B_+^4)$ is noncompact.
Combining this with the conformal covariance~\cite{m21} of the operators $P_4$, $P_3^{S_+^3}$, $P_3^{B^3}$, and $P_2^{S^2}$ yields many solutions to the boundary value problem~\eqref{eqn:bvp} which are not $S^2$-invariant.

\begin{proposition}
 \label{more-solutions}
 Let $u$ be a solution of~\eqref{eqn:bvp}.
 For each $\Phi \in \Conf(B_+^4)$, the function
 \begin{equation*}
  \Phi \cdot u := u \circ \Phi + \log \lvert J_\Phi \rvert^{1/4}
 \end{equation*}
 is also a solution of~\eqref{eqn:bvp}, where $\lvert J_\Phi \rvert$ is the determinant of the Jacobian of $\Phi$.
\end{proposition}

\begin{proof}
 Let $\Phi \in \Conf(B_+^4)$.
 Then there is a $\sigma \in C^\infty(B_+^4)$ such that $\Phi^\ast g = e^{2\sigma}g$.
 By definition of the Jacobian determinant, $e^{2\sigma} = \lvert J_\Phi \rvert^{1/2}$.
 On the one hand, the diffeomorphism invariance of the $Q$-, $T$-, and $U$-curvatures and the computations of Section~\ref{setsec} imply that
 \begin{align*}
  Q_4^{\Phi^\ast g} & = 0 , \\
  T_M^{\Phi^\ast g} & = -2 , \\
  T_N^{\Phi^\ast g} & = 0 , \\
  U^{\Phi^\ast g} & = \frac{\pi}{2} .
 \end{align*}
 The conformal transformation laws for the $Q$-curvature~\cite[p.\ 3679]{Branson1995}, the $T$-curvature~\cite[Lemma~3.3]{cq97i}, and the $U$-curvature~\cite[Theorem~1.1]{m21} then imply that
 \begin{align*}
  P_4\log\lvert J_\Phi\rvert^{1/4} & = 0 , \\
  P_3^M\log\lvert J_\Phi\rvert^{1/4} & = -2\lvert J_\Phi\rvert^{3/4} + 2 , \\
  P_3^N\log\lvert J_\Phi\rvert^{1/4} & = 0 , \\
  P_2\log\lvert J_\Phi\rvert^{1/4} & = \frac{\pi}{2}\lvert J_\Phi\rvert^{1/2} - \frac{\pi}{2} .
 \end{align*}
 Moreover, the diffeomorphism invariance and conformal covariance of the $P_4$-operator~\cite[Theorem~1]{Paneitz}, $P_3$-operator~\cite[Proposition~3.1]{cq97i}, and $P_2$-operator~\cite[Theorem~1.1]{m21} imply that
 \begin{align*}
  \Phi^\ast (P_4v) & = P_4^{\Phi^\ast g}(v \circ \Phi) = \lvert J_\Phi\rvert^{-1} P_4(v \circ \Phi) , \\
  \Phi^\ast (P_3^Mv) & = (P_3^M)^{\Phi^\ast g}(v \circ \Phi) = \lvert J_\Phi\rvert^{-3/4} P_3^M(v \circ \Phi) , \\
  \Phi^\ast (P_3^Nv) & = (P_3^N)^{\Phi^\ast g}(v \circ \Phi) = \lvert J_\Phi\rvert^{-3/4} P_3^N(v \circ \Phi) , \\
  \Phi^\ast (P_2v) & = P_2^{\Phi^\ast g}(v \circ \Phi) = \lvert J_\Phi\rvert^{-1/2} P_2(v \circ \Phi)
 \end{align*}
 for any sufficiently smooth function $v$ on $B_+^4$.
 Therefore
 \begin{align*}
  P_4(\Phi \cdot u) & = \lvert J_\Phi \rvert \Phi^\ast (P_4u) + P_4\log\lvert J_\Phi\rvert^{1/4} = 0 , \\
  P_3^M(\Phi \cdot u) & = \lvert J_\Phi \rvert^{3/4} \Phi^\ast (P_3^Mu) + P_3^M\log\lvert J_\Phi\rvert^{1/4} = 2 , \\
  P_3^N(\Phi \cdot u) & = \lvert J_\Phi \rvert^{3/4} \Phi^\ast (P_3^Nu) + P_3^N\log\lvert J_\Phi\rvert^{1/4} = 0 , \\
  P_2(\Phi \cdot u) & = \lvert J_\Phi \rvert^{1/2} \Phi^\ast (P_2u) + P_2\log\lvert J_\Phi\rvert^{1/4} = \pi e^{2\Phi \cdot u} - \frac{\pi}{2} . \qedhere \\
 \end{align*}
\end{proof}